\documentclass[noinfoline]{article}

\RequirePackage[OT1]{fontenc}
\RequirePackage[
amsthm,amsmath
]{imsart}

\usepackage{graphicx}
\usepackage{latexsym,amsmath}
\usepackage{amsmath,amsthm,amscd}
\usepackage{amsfonts}
\usepackage[psamsfonts]{amssymb}
\usepackage{enumerate}
\usepackage{xcolor}

\usepackage{url}

\usepackage{tocvsec2}

\usepackage{epstopdf}
\usepackage{wrapfig}
\usepackage{float}
\usepackage{hyperref}

\startlocaldefs

\theoremstyle{plain} 
\newtheorem{theorem}{Theorem}

\newtheorem{proposition}[theorem]{Proposition}

\theoremstyle{definition} 

\theoremstyle{definition} 

\newtheorem*{ex*}{Example}
\theoremstyle{remark} 

\theoremstyle{remark} 
\newtheorem{remark}[theorem]{Remark}
\newtheorem*{remark*}{Remark}

\reversemarginpar

\newcommand{\dd}{\partial}

\renewcommand{\dd}{{\,\operatorname{d}}}

\newcommand{\de}{\delta}
\newcommand{\del}{d}

\renewcommand{\Psi}{\overline{\Phi}}

\newcommand{\E}{\operatorname{\mathsf{E}}}

\newcommand{\R}{\mathbb{R}}

\renewcommand{\le}{\leqslant}
\renewcommand{\ge}{\geqslant}

\endlocaldefs

\begin{document}

\begin{frontmatter}

\title{An optimal bound on the quantiles of a certain kind of distributions}
\runtitle{Bound on the quantiles}

%

\begin{aug}
\author{\fnms{Iosif} \snm{Pinelis}\thanksref{t2}\ead[label=e1]{ipinelis@mtu.edu}}
  \thankstext{t2}{Supported in part by NSA grant H98230-12-1-0237
  }
\runauthor{Iosif Pinelis}


\address{Department of Mathematical Sciences\\
Michigan Technological University\\
Houghton, Michigan 49931, USA\\
E-mail: \printead[ipinelis@mtu.edu]{e1}}
\end{aug}

\begin{abstract} 
An optimal bound on the quantiles of a certain kind of distributions is given. 
Such a bound is used in applications to Berry--Esseen-type bounds for nonlinear statistics.  
\end{abstract}

  
%

\setattribute{keyword}{AMS}{AMS 2010 subject classifications:}




\end{frontmatter}

%
%


\theoremstyle{plain} 



Let $\mu$ be any probability measure $\mu$ on $(0,\infty)$. For any real $p$, let 
\begin{equation}\label{eq:mu_p}
	\mu_p:=\int_{(0,\infty)}x^{p-2}\,\mu(\dd x). 
\end{equation}
Consider the function $L\colon(0,\infty)\to\R$ defined by the formula 
\begin{equation*}
L(d):=L_\mu(d):=\int_{(0,\infty)}(1\wedge\tfrac\del x)\,\mu(\dd x). 	
\end{equation*}
Clearly, $L$ is continuous and nondecreasing, with $L(0)=0$ and $L(\infty-)=1$. 

Take now an arbitrary $c\in(0,1)$. Then the equation 
\begin{equation}\label{eq:L(d)=c}
L(d)=c	
\end{equation}
has a root $d\in(0,\infty)$. Moreover, this root is unique. Indeed, if 
$
L(d)=c	
$
for some $d\in(0,\infty)$, then $\int_{(\del,\infty)}\tfrac{x-d}x\,\mu(\dd x)=L(\infty-)-L(d)=1-c>0$ and hence $\int_{(\del,\infty)}\mu(\dd x)>0$ and the right derivative of the function of $L$ at the point $d$ is $\int_{(\del,\infty)}\tfrac1x\,\mu(\dd x)>0$. 
So, the definition 
\begin{equation}\label{eq:de}
	\de:=\de_\mu:=L^{-1}(c)=L_\mu^{-1}(c)
\end{equation}
is proper. 
Note that $\de$ may be viewed as the $c$-quantile of the distribution function $L$ of a probability distribution on $(0,\infty)$. 

\begin{theorem}\label{th:}
If $\mu_3<\infty$, then 
\begin{equation}\label{eq:}
\de\le\de_*:=
\left\{
\begin{alignedat}{2}
&c\mu_3 &\text{\quad if \quad}&0<c\le\tfrac12, \\ 
&\frac{\mu_3-(2c-1)^2/\mu_1}{4(1-c)} &\text{\quad if \quad}&\tfrac12\le c<1.  
\end{alignedat}
\right.
\end{equation} 
\end{theorem} 
Note that the two expressions for $\de_*$ in \eqref{eq:} in the case $c=\frac12$ have the same value, $\mu_3/2$.

\begin{proof}[Proof of Theorem~\ref{th:}] 
Consider first the case $\tfrac12<c<1$. 
Take any real numbers $\del$ and $u$ such that $0<u<\del$ and introduce also the  
functions $f$ and $g$ on $(0,\infty)$ defined by the formulas 
\begin{equation*}
	f(x):=(1\wedge\tfrac\del x)-c\quad\text{and}\quad
	g(x):=b_0-b_3x-\tfrac{b_1}x
\end{equation*}
for $x>0$, 
where  
\begin{equation*}
	b_0:=\tfrac{2(1-c)\del +(2c-1) u}{2(\del - u)},\quad b_3:=\tfrac1{4(\del - u)},\quad b_1:=\tfrac{u^2}{4(\del - u)};  
\end{equation*}
note that 
$b_1>0$, so that $g$ is strictly 
concave on $(0,\infty)$. 
Let also 
\begin{equation}\label{eq:v}
	v:=2\del-u. 
\end{equation}
Then $v\in(\del,\infty)$, and one can check that $(f-g)(w)=(f-g)'(w)=0$ for $w\in\{u,v\}$. 
Since the function $f-g$ is strictly 
convex on $(0,\del]$ and on $[\del,\infty)$, it follows that $f>g$ 
on $(0,\infty)\setminus\{u,v\}$ and $f=g$ on $\{u,v\}$. 
So, 
\begin{equation}\label{eq:chain}
\begin{aligned}
	L(\del)-c
	&=\int_{(0,\infty)}f\dd\mu \\
	&\ge\int_{(0,\infty)}g\dd\mu
	=b_0-b_3\mu_3-b_1\mu_1
	=\tfrac{1-c}{\del-u}\,\big(\del-\del_*(u)\big)=0
\end{aligned}	
\end{equation}
if $\del=\del_*(u)$, where 
\begin{equation}\label{eq:d(u_*)}
	\del_*(u):=\tfrac{\mu_3 - 2u (2c-1) + \mu_1 u^2}{4(1-c)}. 
\end{equation}
Next, 
\begin{equation*}
\del_*(u)\ge\del_*(u_*)=\de_*, 	
\end{equation*}
where
\begin{equation}\label{eq:u_*}
	u_*:=\tfrac{2c-1}{\mu_1}.  
\end{equation}
Obviously, $u_*>0$. 
Also, $\mu_p$ (defined by \eqref{eq:mu_p}) is log-convex in $p>0$ and hence 
\begin{equation}\label{eq:log-c}
	\mu_3\mu_1\ge\mu_2^2=1. 
\end{equation}
So, $\de_*-u_*=\frac{4 (1-c)^2+\mu_1 \mu_3-1}{4 (1-c) \mu_1}>0$, and hence $0<u_*<\de_*$. 
Thus, $L(\de_*)=L\big(\del_*(u_*)\big)\ge c$, and the inequality $\de\le\de_*$ in \eqref{eq:} in the case $\tfrac12<c<1$ follows by the monotonicity of the function $L$. 

The case $0<c\le\tfrac12$ is similar and even simpler. Take here $\del=\de_*=c\mu_3$ and $g(x):=c-cx/\mu_3$ for $x>0$. Then 
$(f-g)(0+)=1-2c\ge0$ and 
$(f-g)(\mu_3)=(f-g)'(\mu_3)=0$. 
Since the function $f-g$ is strictly 
convex on $[\del,\infty)$ and affine on $(0,\del]$, it follows that $f>g$ 
on $(0,\infty)\setminus\{\mu_3\}$ and $f=g$ on $\{\mu_3\}$. 
So, 
\begin{equation}\label{eq:chain2}
	L(c\mu_3)-c
	=\int_{(0,\infty)}f\dd\mu 
	\ge\int_{(0,\infty)}g\dd\mu=0,  	
\end{equation}
and the inequality $\de\le\de_*$ in \eqref{eq:} in the case $0<c\le\tfrac12$ follows by the monotonicity of $L$.
\end{proof}

\begin{remark}\label{rem:strict}
It is clear from the above proof of Theorem~\ref{th:} that the inequality $\de\le\de_*$ in \eqref{eq:} 
is strict unless the support of the measure $\mu$ consists of one point (in the case $0<c\le\tfrac12$) or of two points (in the case $\tfrac12<c<1$).  
\end{remark}

Moreover, the upper bound $\de_*$ on $\de$ is the best possible one in terms of $c$, $\mu_3$, and $\mu_1$ in the following sense: 
\begin{proposition}\label{prop:best}\ 
\begin{enumerate}[(I)]
	\item For any $c\in(0,\frac12]$ and any positive real number $\mu_{3*}$, there exists a probability measure $\mu$ on $(0,\infty)$ such that $\de=\de_*$ and \eqref{eq:mu_p} holds for $p=3$ with $\mu_{3*}$ in place of $\mu_3$. 
	\item For any $c\in(\frac12,1)$ and any positive real numbers $\mu_{3*}$ and $\mu_{1*}$ such that $\mu_{3*}\mu_{1*}\ge1$ \big(cf.\ \eqref{eq:log-c}\big), there exists a probability measure $\mu$ on $(0,\infty)$ such that $\de=\de_*$ and \eqref{eq:mu_p} holds for $p\in\{1,3\}$ with $\mu_{3*}$ and $\mu_{1*}$ in place of $\mu_3$ and $\mu_1$. 
\end{enumerate} 
\end{proposition}

\begin{proof}[Proof of Proposition~\ref{prop:best}] \ Let us consider first the more complicated case (II).   

(II). 
Take indeed any $c\in(\frac12,1)$ and any positive real numbers $\mu_{3*}$ and $\mu_{1*}$ such that $\mu_{3*}\mu_{1*}\ge1$. 
Let $\de_*$ and $u_*$ be defined as in \eqref{eq:} and \eqref{eq:u_*}, respectively, but with $\mu_{3*}$ and $\mu_{1*}$ in place of $\mu_3$ and $\mu_1$. 
As shown in the proof of Theorem~\ref{th:}, $0<u_*<\de_*$. 
Now, in accordance with \eqref{eq:v} and \eqref{eq:d(u_*)}, introduce $v_*:=2\de_*-u_*$; it follows that $v_*>\de_*$. 
Let $\mu$ be the probability measure with masses $1-\pi$ and $\pi$ at the points $u_*$ and $v_*$, respectively, where $\pi:=\frac{2 (1-c) (\mu_{3*}\mu_{1*}-(2 c-1))}{4(1-c)^2+\mu_{3*}\mu_{1*}-1}$; note that such a measure $\mu$ exists, since $0<\pi\le1$. 
Moreover, one can check that then \eqref{eq:mu_p} holds for $p\in\{1,3\}$ with $\mu_{3*}$ and $\mu_{1*}$ in place of $\mu_3$ and $\mu_1$.  
Recalling now that $f=g$ on the set $\{u,v\}$ and using \eqref{eq:chain} with $\de_*=\del_*(u_*)$, $u_*$, $\mu_{3*}$, and $\mu_{1*}$ in place of $d$, $u$, $\mu_3$, and $\mu_1$, one concludes that $\de_*=\frac1{4(1-c)}\big(\mu_{3*}-\frac{(2c-1)^2}{\mu_{1*}}\big)$ is indeed a positive real root $d$ of the equation \eqref{eq:L(d)=c}. 
Finally, the uniqueness of such a root was established in the paragraph containing \eqref{eq:de}.  

(I). 
The case $c\in(0,\frac12]$ is similar but simpler. Here we let $\mu$ be the Dirac probability measure with mass $1$ at the point $\mu_3$. Then the inequality in \eqref{eq:chain2} turns into the equality. 
\end{proof} 

Take now any natural $n$ and let $\xi_1,\dots,\xi_n$ be any 
random variables 
such that 
$
	\E\xi_1^2+\dots+\E\xi_n^2=1.  
$ 
Let then $\mu_\xi$ be the 
probability measure on $(0,\infty)$ defined by the condition 
$
	\int_{(0,\infty)}h\dd\mu=\sum _{i=1}^n \E h(|\xi_i|)\xi_i^2
$ 
for all (say) nonnegative Borel functions $h$ on $[0,\infty)$. 

For $\mu=\mu_\xi$ and the ``median'' value $c=\frac12$, the upper bound $\de_*=\mu_3/2$ on $\de$ follows immediately from the inequality due to Chen and Shao \cite[Remark 2.1]{chen07}, who showed that 
\begin{equation}\label{eq:chen-shao}
	\de\le\Big(\frac{2(p-2)^{p-2}}{(p-1)^{p-1}}\mu_p\Big)^{\frac1{p-2}} 
\end{equation}
for $p>2$. 
On the other hand, the bound $\de_*$ in \eqref{eq:} is more general than the one in \eqref{eq:chen-shao} in the sense that $c$ in \eqref{eq:} is allowed to take any value in the interval $(0,1)$; this flexibility   
allows one to improve the corresponding results in \cite{nonlinear}.

\bibliographystyle{abbrv}


\bibliography{C:/Users/iosif/Dropbox/mtu/bib_files/citations12.13.12}


\end{document}